\documentclass[graybox]{svmult}

\usepackage{mathptmx}       
\usepackage{helvet}         
\usepackage{courier,amsfonts,amsmath}        
\usepackage{type1cm}        
\usepackage[leftcaption]{sidecap}
\sidecaptionvpos{figure}{m}
\usepackage{tikz}

\usepackage{makeidx}         
\usepackage{graphicx}        
\usepackage{multicol}        
\usepackage[bottom]{footmisc}

\makeindex             

\begin{document}

\title*{Loose ends in a strong force 3-body problem}
\author{Connor Jackman}
\institute{Connor Jackman \at CIMAT, Guanajuato, Mexico \email{connor.jackman@cimat.mx}
}
%
%
\maketitle

\abstract*{Up to symmetries, the orbits of three equal masses under an inverse cube force with zero angular momentum and constant moment of inertia can be reparametrized as the geodesics of a complete, negatively curved metric on a pair of pants. The ends of the pants represent binary collisions. Here we will examine the visibility properties of such negatively curved surfaces, allowing a description of orbits beginning or ending in binary collisions of this 3-body problem.}

\abstract{Up to symmetries, the orbits of three equal masses under an inverse cube force with zero angular momentum and constant moment of inertia can be reparametrized as the geodesics of a complete, negatively curved metric on a pair of pants. The ends of the pants represent binary collisions. Here we will examine the visibility properties of such negatively curved surfaces, allowing a description of orbits beginning or ending in binary collisions of this 3-body problem.}

\section{Introduction}
\label{sec:Intro}

It was noted by Poincar\'e \cite{P} that $N$ point masses subject to an attractive force proportional to the inverse $a^{th}$-power of the mutual distances are especially suited to variational methods when $a\ge 3$, often called \textit{strong force} $N$-body problems. The simplification comes from observing that for such forces, the action of a path passing through a collision is infinite or similarly that the Jacobi-Maupertuis metric (or JM metric for short, see eq. \ref{JM} below) is complete. Consequently, there are less obstacles to applying the direct method: over a class of curves having \textit{finite} action -- provided a minimizing sequence converges to some curve -- an action minimizing curve is collision free, its action being finite. For example, minimizing over certain 'tied' free homotopy classes of curves, one can describe a plethora of periodic orbits in these strong force problems. Our main result here is that for the inverse cube force one may, via the JM-metric, understand certain orbits having binary collisions as well.

We consider the planar three body problem under an inverse cube force -- which has some exceptional properties (see e.g. \cite{AlbHom}). For this strong force, the Lagrange-Jacobi identity (eq. \ref{LJ}) shows that periodic orbits are only possible at the zero energy level, which is the motivation in \cite{MPants, shirts} for studying orbits with zero energy. Although collision orbits occur also on the non-zero energy levels, our focus here is to complete the description of orbits on this zero energy level. The Jacobi-Maupertuis principle allows one to reparametrize orbits of a natural Hamiltonian system on a fixed energy level as geodesics of a certain metric -- the JM-metric -- defined on the configuration space $Q$. The symmetry group $G$ of the inverse cube problem consists of translations, rotations \textit{and scalings} of the triangle formed by the three bodies and are now isometries of the zero energy JM-metric. We may, by Riemannian submersion, define a reduced metric on the quotient $Q/G=:\Sigma$. Due to the additional scaling symmetry this quotient space is two dimensonal, topologically it is a sphere minus 3 points, or a pair of pants  (see figure \ref{fig:pants}). Geodesics of the reduced JM-metric on $\Sigma$ represent zero energy orbits up to symmetries of the inverse cube 3-body problem which move perpendicularly to the $G$-orbit at each instant.

The advantages of this process for the inverse cube 3-body problem are illustrated in Montgomery's article \cite{MPants}. Montgomery computed that, when the three masses are equal, the Gaussian curvature of the JM-metric on this pair of pants is negative away from a discrete set. This allows one to describe all such periodic orbits by the free homotopy class they realize on $\Sigma$ -- the negative curvature allowing one to assert that the correspondence is one to one: up to symmetries, there is \textit{at most} one periodic orbit in each free homotopy class of $\Sigma$. On pg. 6 of \cite{MPants}, Montgomery leaves some open questions or 'loose ends', asking whether one can likewise code the orbits beginning and ending in collisions -- in particular the action or JM-length of such orbits is infinite. In this article, we will tie up these loose ends by describing the geodesics on $\Sigma$ which begin or end in binary collisions (theorem \ref{thm:syz} below). We describe these orbits using 'syzygy sequences':

\begin{SCfigure}[2][b]
    \begin{tikzpicture}[scale = .7]

  \draw[color=black] (-1,2) to [bend right=20] (2,.3);
  
  \node[below] at (.5,1.5) {3};

  \draw[color=black] (2,-.3) to [bend right=20] (-1,-2);
  
  \node at (.5,-1.1) {2};
  
  \draw[color=black] (-1.5,-1.9) to [bend right=20] (-1.5,1.9);
  
  \node at (-1.5,0) {1};
  
  \draw [thick, color = red] (-1.1,.1) to[out=70,in= -185] (-.5,.5) to[out= -5,in=-185] (2,.1);
  
  \draw [dashed, thick, color = red] (-1.1,.1) to[out=-105,in=180] (-.4,-1) to[out=0,in=-135] (.245,-.85);
  
  \draw [thick, color = red] (.245,-.85) to[out=45,in=-40] (.2,.3) to[out=140,in=-50] (-1.3,2);

\end{tikzpicture}
\caption{The pair of pants $\Sigma$ and a collision orbit (red) realizing the syzygy sequence $12$. The 3 collinear arcs (black) are labelled $1,2,3$ and divide $\Sigma$ into an upper and lower region -- these two regions are related by reflecting the planar configuration, which is a symmetry of $\Sigma$.}
\label{fig:pants}
    \end{SCfigure}
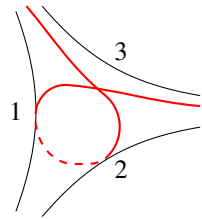

\begin{definition}
Label the 3-bodies by 1,2,3 and the collinear arcs on $\Sigma$ by which body is in the middle. A \textit{syzygy sequence} is a map, $s$, from $I\subset\mathbb{Z}$ to $\{1,2,3\}$, i.e. a list of the symbols $1,2,3$. We call a syzygy sequence \textit{finite} when $I=\{1,2,...,N\}$,  \textit{semi-infinite} when  $I=\mathbb{N}$ and \textit{bi-infinite } when $I=\mathbb{Z}$, such sequences are said to be \textit{stutter free} if $s(i)\ne s(i+1)$.
\end{definition}

To any curve on $\Sigma$, we may assign a syzygy sequence by listing in temporal order the collinear arcs crossed by the curve. One may 'homotope away' any tangencies to the collinear arcs or stutters in a given syzygy sequence. For example a curve with syzygy sequence $1233$ is homotopic to a curve with syzygy sequence $12$ and for the collinear arcs themselves, one may assign either the bi-infinite $...aaa...$ for $a\in\{1,2,3\}$ or, by cancelling stutters, the empty sequence. We always assign a closed curve its bi-infinite (repeating) syzygy sequence, which can be represented with an overbar, for example $\overline{12}$ represents a loop around one of the ends.

\begin{definition}
By a \textit{collision orbit} of the planar 3-body problem, we mean a solution $(q_1(t),q_2(t),q_3(t))\in\mathbb{C}^3$ s.t. $|q_i(t)-q_j(t)|\to 0$ as $t\to t_c$ for some $i\ne j$ and $t_c\in\mathbb{R}$. We call a collision orbit of the planar 3-body problem a \textit{straight collision orbit} if its projection to $\Sigma$ has finite syzygy sequence, and a \textit{winding collision orbit} if its projection to $\Sigma$ begins and ends with a sequence of two alternating symbols, e.g. $...121212,31,323232...$.
\end{definition}

\begin{theorem}\label{thm:syz}
Consider the planar inverse cube three body problem with equal masses. Up to symmetries, orbits with zero angular momentum and constant moment of inertia are reparametrized as geodesics on the surface $\Sigma$. Then:

(i) any finite stutter free syzygy sequence is realized by two geodesics (straight collision orbits).

(ii) the stutter free syzygy sequences of the form $...abababs_1...s_kcdcdcd...$ with $s_1\ne a$, $s_k\ne d$ are all realized by multiple geodesics (1-parameter families of winding collision orbits).

\end{theorem}

\begin{remark}\label{rmk:symms}
The two straight collision orbits realizing a given syzygy sequence are related by the symmetry of $\Sigma$ induced by a reflection in the plane.
\end{remark}

\begin{remark}\label{rmk:shirts}
In \cite{shirts}, we show that the dynamics of parallelogram configurations in the equal masses 4-body problem under an inverse cube force can also be reduced to a non-positively curved geodesic flow on a 'shirt' or sphere with 4 punctures. The proof of theorem \ref{thm:syz} goes through without significant differences to describe the collision orbits in this parallelogram problem as well.
\end{remark}

\begin{remark}[Further loose ends]\label{rmk:nonzero}
When considering zero angular momentum periodic orbits of the equal masses inverse cube problem, there is no loss of generality in taking the constraints imposed by the hypotheses of theorem \ref{thm:syz}: every periodic orbit has constant moment of inertia and zero energy. For collision orbits, these constraints are not so natural. In particular it would be interesting to see if the methods here can be applied to describe collision or escape orbits with non-zero energy.
\end{remark}

The proof of theorem \ref{thm:syz} boils down to verifying some 'visibility properties' on $\Sigma's$ universal cover, $H$: given lifts $\tilde\gamma_1,\tilde\gamma_1\in H$ of geodesics on $\Sigma$, when does there exist a geodesic forward asymptotic to $\tilde\gamma_1$ and backwards asymptotic to $\tilde\gamma_2$? The result on collision orbits amounts to the statement that $\Sigma$ is 'visible with respect to the collinear arcs'. We show this using Busemann functions. After this, the uniqueness follows from Toponogov's theorem, and the description of winding orbits from perturbing the straight collision orbits.

In section \ref{sec:not} we set up the problem -- defining the reduced JM-metric on the pair of pants $\Sigma$, and in section \ref{sec:vis} we recall the relevant notions of visibility manifolds (see \cite{EOVis}) used to prove theorem \ref{thm:syz} in section \ref{sec:prf}. In fact we prove a slightly more general visibility property of certain non-positively curved metrics on spheres with $k\ge 3$ punctures (lemma \ref{Lem:vis}).



\section{The reduced JM-metric on $\Sigma$}
\label{sec:not}

Identifying the plane with the complex numbers, the configuration space for 3 point masses in the plane is $$Q:=\mathbb{C}^3\backslash \Delta,$$ where $\Delta :=\{ (q_1,q_2,q_3)\in Q : q_j = q_k$ for some $j\ne k\}$ consists of the \textit{collisions}.

The potential for three \textit{unit} masses under an inverse cube force is\\ $U := \sum_{j<k}|q_j-q_k|^{-2}$, and we may write the equations of motion as: \begin{equation}\label{eq:motion}
    \ddot q_j = \frac{\partial U}{\partial q_j}.
\end{equation} 
Also, one has that the \textit{energy}, $E:= \frac{\sum_{j=1}^3 |\dot q_j|^2}{2} - U(q)$, is constant over solutions of eq. \ref{eq:motion}.


The Jacobi-Maupertuis principle (see \cite{Arn} \S 45D), states that the solutions of eq. \ref{eq:motion} at a fixed energy level $E^{-1}(e)$ can be reparametrized as geodesics of the JM-metric:\begin{equation}\label{JM}
    ds_{JM}^2 := (e+U)ds^2,
\end{equation} where $ds^2 := \sum_{j=1}^3dq_jd\overline q_j$ is the standard Euclidean metric on $\mathbb{C}^3$. The JM-metric is defined on the Hill region: $\{ q: e+U(q)>0\}\subset Q$. 

The symmetry on solutions of eq. \ref{eq:motion} under translations and boosts, allows to carry out the translation reduction by the choice of in inertial frame with center of mass zero. That is, we restrict to solutions lying in $$Q_0:=\{ (q_1,q_2,q_3)\in Q : \sum q_j = 0\}\cong \mathbb{C}^2\backslash \Delta_0,$$ where $\Delta_0$ consists of 3 complex lines through the origin of $\mathbb{C}^2$. On $Q_0$, the moment of inertia is given by $I(q):=\sum_{j=1}^3 q_j\overline{q}_j$. Over a solution $q(t)\in Q_0$ with energy $e$, due to $U$'s homogeneity of degree $-2$, we have the Lagrange-Jacobi identity \begin{equation}\label{LJ}
    \ddot I = 4e
\end{equation} In particular, periodic orbits are only possible for zero energy which is the motivation in \cite{MPants, shirts} for fixing attention to the zero energy level.

The zero energy JM-metric, $Uds^2$ on $Q_0$ is invariant under complex scaling. The quotient map $\pi: Q_0\to Q_0 / \mathbb{C}^*, q_0\mapsto [q_0]$ is, under a linear identification of $Q_0$ with $\mathbb{C}^2\backslash\Delta_0$ ,the usual Hopf map so that $$Q_0/\mathbb{C}^*\cong S^2\backslash\{3pts\}.$$ Now, since scaling is a symmetry of the zero energy JM-metric we may define a metric, $d\overline{s}_{JM}^2$, on the quotient by $$d\overline{s}_{JM[q_0]}^2(\pi_*u, \pi_*v):= ds_{JMq_0}^2(u,v).$$ The geodesics of $\Sigma:=Q_0/\mathbb{C}^*$ under the metric $d\overline{s}_{JM}^2$, represent zero energy solutions $q(t)$ of eq. \ref{eq:motion} up to symmetries moving perpendicular to the fibers: $$0=ds^2(\dot q, iq) = C,~~ 0 = ds^2(\dot q, q) = \dot I$$ where $C$ is the angular momentum of the solution, and by eq. \ref{LJ} the condition $\dot I = 0$ along with $E= 0$ are equivalent to the condition that moment of inertia be constant over the solution.

\section{Visibility manifolds}
\label{sec:vis}

We recall some notions of hyperbolic geometry (see e.g. \cite{EOVis, B}) that allow us to prove Lemma \ref{Lem:vis} -- which will be the main tool used to construct collision orbits on the pair of pants associated to the reduced strong force 3-body problem.

Let $M$ be a complete non-positively curved surface, then $M$ has no conjugate points and the exponential map at a point is a covering map -- the universal cover, $H$, of $M$ is topologically  $\mathbb{R}^2$ and we may pull back the metric on $M$ to equip $H$ with a complete non-positively curved metric ($H$ is called a \textit{Hadamard manifold}). We always consider \textit{unit speed} geodesics on $H$. Two geodesics $\alpha,\beta$ of $H$ are \textit{forward asymptotic} (resp. \textit{backwards asymptotic}) if $d(\alpha(t), \beta(t)) = O(1)$ as $t\to \infty$ (resp. $t\to-\infty$), where $d$ is the distance function induced by the metric on $H$. Forward asymptotic is an equivalence relation on geodesics of $H$ and we write $H(\infty)$ for the set of equivalence classes, and $\alpha(\infty)$, (resp. $\alpha(-\infty)$), for the class of geodesics forward asymptotic to $\alpha(t)$, (resp. $\alpha(-t)$). For two points $x\ne y\in H(\infty)$ we would like to determine when there exists a geodesic $\alpha$ of $H$ from $x$ to $y$, i.e. with $\alpha(\infty)=x$ and $\alpha(-\infty) = y$.

\begin{definition}
A non-positively curved manifold $M$ is \textit{visible with respect to} the geodesics $\gamma_1, \gamma_2$ of $M$ if for any lifts, $\tilde\gamma_i$, of $\gamma_i$ to $H$, and choice of distinct points $x,y\in \{ \tilde\gamma_i(\pm\infty)\}$, there exists a geodesic of $H$ from $x$ to $y$.
\end{definition}

We now recall some useful properties of Busemann functions. A \textit{Busemann function} for $x = \alpha(\infty)\in H(\infty)$ is $f_x(h):= \lim_{t\to\infty} (d(h,\alpha(t)) - t)$, this function $f_x:H\to\mathbb{R}$ being well defined up to shifts by a constant. Hence the foliation of $H$ into level sets $f_x^{-1}(c)$, called \textit{horocycles of } $x$, does not depend on the representative chosen for $x$. It can be shown (see \cite{EOVis} pg. 58) that the function $f_x$ is smooth and that its gradient $\nabla f_x(h)$ gives the initial velocity of a geodesic forward asymptotic to $x$. In particular it follows that:
\begin{description}[Property]
\item[Property 1]{For $x\ne y\in H(\infty)$, if there exist disjoint horocycles of $x$ and $y$ ($f_x^{-1}(c_1)\cap f_y^{-1}(c_2) =\emptyset$ for some $c_i\in\mathbb{R}$), then there exists a geodesic from $x$ to $y$.}
\end{description}
Which can be seen by fixing $c_1$ and considering the first value $c\in\mathbb{R}$ for which $f_y^{-1}(c)\cap f_x^{-1}(c_1)\ne\emptyset$. At a point $h$ in this intersection, the two horocycles are tangent and a geodesic with initial velocity $\nabla f_y(h)$ will connect $x$ to $y$. We will also make use of (see \cite{EOVis} pg. 57):

\begin{description}[Property]
\item[Property 2]{Horocycles of $x$ have: $d(f_x^{-1}(c_1), f_x^{-1}(c_2)) = |c_1 - c_2|$.}
\end{description}

Now let $P_k$ be homeomorphic to  $S^2\backslash \{p_1,...,p_k\}$ -- a sphere with $k\ge 3$ punctures and equipped with a complete metric of non-positive curvature. We say $P_k$ has \textit{finite diameter ends} if for each $p_j$ we have $\sup_U \{\inf length(\gamma)\}<\infty$ where $U$ is a neighborhood of $p_j$ and $\gamma$ a loop in $U$ realizing the free homotopy class of a loop around $p_j$. We can show:

\begin{lemma}\label{Lem:vis}
Suppose $P_k\cong S^2\backslash\{ p_1,...,p_k\}$ is equipped with a complete non-positively curved metric having finite diameter ends and for which there exist $k$ disjoint geodesics ('seams') $\gamma_j$ from $p_j$ to $p_{j+1}$, for $j=1,...,k$ (and $p_{k+1}:=p_1$). Then $P_k$ is visible with respect to $\gamma_j$.
\end{lemma}

\begin{proof}
\smartqed
Opening $P_k$ along the seams $\gamma_1,...,\gamma_{k-1}$ we have a simply connected region $D$, whose lifts (fundamental domains) tile $H$. Consider a lift $\tilde D\subset H$ of $D$, then $H\backslash \tilde D$ consists of $2(k-1)$ connected components (see figure \ref{fig:fd} for labeling). The key observation is that for $x \in\{\gamma_j^{\tilde D}(\pm\infty)\}$ there are horocycles contained in $\tilde D$ and the components of $H\backslash\tilde D$ 'adjacent to $x$'. For example there are horocycles of $\gamma_1^{\tilde D}(\infty)$ contained in $\tilde D\cup \tilde D_1\cup \tilde D_2$.

\begin{figure}[!htb]
    \centering
   
         \begin{tikzpicture}[scale = .5]

         \path [fill=lightgray] (-10,-.6) to[out=0,in= 92] (-6.3,-3.2) -- (-5.7,-3.2) to[out=88,in= -180] (-4,-.6) to[out= 0,in=92] (-2.3,-3.2) -- (2.3,-3.2) to[out=88,in= -180] (4,-.6) to[out= 0,in=92] (5.7,-3.2) -- (6.3,-3.2) to [out =88, in = -180] (10,-.6) -- (10,.6) to [out = -180, in = -88] (6.3,3.2) -- (5.7,3.2) to[out=-92,in= 0] (4,.6) to[out= 180,in=-88] (2.3,3.2) -- (-2.3,3.2) to[out=-92,in= 0] (-4,.6) to[out= 180,in=-88] (-5.7,3.2) -- (-6.3,3.2) to[out=-92,in= 0] (-10,.6) -- (-10,-.6);
         
         \draw[color = black, <-] (-10,0) -- (-1,0);
         
         \draw[color = black] (1,0) -- (10,0);
         
         \node at (0,0) {\textbf{.~.~.}};
         
         \node[above] at (0,1) {$\tilde D$};
         
         \node[above] at (1.5,0) {$\gamma_k^{\tilde D}$};
         
         
         \draw [color = black, ->] (-10,-.6) to[out=0,in= 92] (-6.3,-3.2);
         
         \node at (-7,-.9) {$\gamma_1^{\tilde D}$};
         
         \node at (-9,-2.8) {$\tilde D_1$};

         \draw [color = black, ->] (-5.7,-3.2) to[out=88,in= -180] (-4,-.6) to[out= 0,in=92] (-2.3,-3.2);
         
         \node at (-2.5,-.9) {$\gamma_2^{\tilde D}$};
         
         \node at (-4,-2.8) {$\tilde D_2$};

         \draw [dashed, color = black, ->] (-10,.6) to[out=0,in= -92] (-6.3,3.2);
         
         \node at (-7,.9) {$\gamma_{-1}^{\tilde D}$};
         
         \node at (-9,2.8) {$\tilde D_{-1}$};

         \draw [dashed, color = black,->] (-5.7,3.2) to[out=-88,in= -180] (-4,.6) to[out= 0,in=-92] (-2.3,3.2);
         
         \node at (-2.5,.9) {$\gamma_{-2}^{\tilde D}$};
         
         \node at (-4,2.8) {$\tilde D_{-2}$};

         \draw [color = black,->] (2.3,-3.2) to[out=88,in= -180] (4,-.6) to[out= 0,in=92] (5.7,-3.2);
         
         \node at (5.8,-.9) {$\gamma_{k-2}^{\tilde D}$};
         
         \node at (4,-2.8) {$\tilde D_{k-2}$};

         \draw[color=black,->] (6.3,-3.2) to [out =88, in = -180] (10,-.6);
         
         \node at (9.5,-1) {$\gamma_{k-1}^{\tilde D}$};
         
         \node at (9,-2.8) {$\tilde D_{k-1}$};

         \draw [dashed, color = black,->] (2.3,3.2) to[out=-88,in= -180] (4,.6) to[out= 0,in=-92] (5.7,3.2);
         
         \node at (5.8,.9) {$\gamma_{2-k}^{\tilde D}$};
         
         \node at (4,2.8) {$\tilde D_{2-k}$};

         \draw[dashed, color=black,->] (6.3,3.2) to [out = -88, in = -180] (10,.6);
         
         \node at (9.5,1) {$\gamma_{1-k}^{\tilde D}$};
         
         \node at (9,2.8) {$\tilde D_{1-k}$};

         \draw[thick, color=red] (-8,-3.2) to [out = 30, in = 120] (-5,-3.2);
         

\end{tikzpicture}
        \caption{A fundamental domain $\tilde D$ with labeled edges and components of $\tilde D^c$ -- we will use this same labeling convention for a general fundamental domain. In red is a horocycle of $\gamma_1^{\tilde D}(\infty) = \gamma_2^{\tilde D}(-\infty)$.}
        \label{fig:fd}
\end{figure}
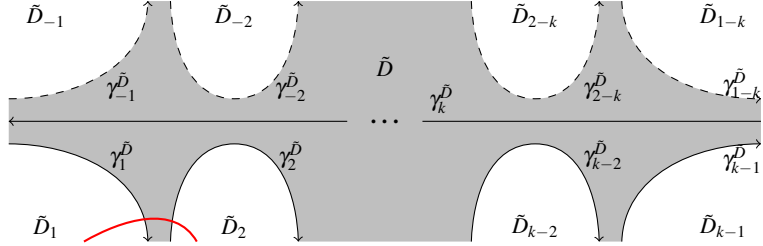


Indeed, let $x=\gamma_1^{\tilde D}(\infty)$. Since the ends are finite diameter, we have $x=\gamma_2^{\tilde D}(-\infty)$, and may choose a Busemann function $f_x$ s.t. $f_x(\gamma_1^{\tilde D}(s)) = -s$ and $f_x(\gamma_2^{\tilde D}(-s)) = -s + cst.$ for $s\in\mathbb{R}$.  Hence the horocycle $f_x^{-1}(-s)$ crosses each of $\gamma_{1}^{\tilde D}, \gamma_{2}^{\tilde D}$ in exactly one point, in particular it consists of two rays $r_1, r_2$ contained in $\tilde D_1, \tilde D_2$ respectively and a smooth arc $h_s$ connecting $\gamma_1^{\tilde D}(s)$ to $\gamma_2^{\tilde D}(-s + cst.)$ and contained in $(\tilde D_1\cup \tilde D_2)^c$. For given $s>0$, the arc $h_s$ may not be contained entirely in $\tilde D$: it is possible $h_s$ wanders into some $\tilde D_j$ ($j\ne 1,2$) for some time before returning to $\tilde D$ (in order to terminate at $\gamma_2^{\tilde D}(-s + cst.)$). However, by property 2, $d(h_s, h_{s+\delta}) = \delta$ and so by taking $\delta$ sufficiently large, we may seperate $h_{s+\delta}$ from any of these excursions of $h_s$ into $\tilde D_j$ -- in particular $h_{s+\delta}\subset \tilde D$ for $\delta$ sufficiently large.

The main idea of the proof now is that the points of $H(\infty)$ coming from lifts of seams that we wish to connect are either already connected by a seam or their horocycles have a seperating strip between all but a compact arc, property 2 allowing us to seperate these horocycles and apply property 1. To understand the notation, the reader may wish to consider when $k=3$ in what follows.

Without loss of generality, we will show that $P_k$ is visible wrt $\gamma_1, \gamma_j$ for $j=1,...,k$ and consider a fixed lift, $\tilde\gamma_1=\gamma_1^{\tilde D}$ of $\gamma_1$ lying in fundamental domain $\tilde D$.

case1: $\tilde\gamma_j\in cl(\tilde D)$. Distinct points of $x,y\in \{ \tilde\gamma_1(\pm\infty), \tilde\gamma_j(\pm\infty)\}$ which are not already connected by a seam eventually have horocycles lying in regions which, apart from possible overlap in $\tilde D$, are disjoint. By property 2, fixing such a horocycle through $x$ say and taking $c<<0$, we ensure that $f_y^{-1}(c)$ becomes disjoint from this horocycle through $x$. Hence, by property 1, there exists a geodesic from $x$ to $y$. 

case2: $\tilde\gamma_j\notin cl(\tilde D)$. Then $\tilde\gamma_j$ lies in a fundamental domain $\tilde E\ne \tilde D$. To $\tilde E$ is associated a unique finite list of the symbols $\{ \pm 1,...,\pm (k-1)\}$: we choose $a_1$ s.t. $\tilde D_{a_1}\supset \tilde E$ and let $\tilde E^1\subset \tilde D_{a_1}$ be the fundamental domain bordering $\tilde D$ along $\gamma_{a_1}^{\tilde D}$, then choose $a_2$ s.t. $\tilde E_{a_2}^1\supset \tilde E$, and so on until $\tilde E^n=\tilde E$. If this sequence begins with $a_1$ then for any seam lifted to $\tilde E$ its horocycles are either equivalent to those of $\gamma_{a_1}^{\tilde D}(\pm\infty)$ (treated in case 1) or eventually are entirely contained in $\tilde D_{a_1}$. Hence for $a_1\notin\{\pm 1,2 \}$ we have disjoint horocycles and connecting geodesics.

It remains to consider when the sequence of $\tilde E$ begins with some number $n$ of the symbols $\{\pm 1, 2\}$ before terminating or taking another symbol. If $n=1$ we may argue as when the sequence of $\tilde E$ begins with $a_1\notin \{\pm 1, 2\}$, so consider $n\ge 2$, say we begin with $1,1,...$ (the other possibilities can be handled in the same way). For a seam lifted to $\tilde E$, it has either some horocycles contained in $\tilde E_1^1$, or equivalent to those of $\gamma_1^{\tilde E^1}(\pm\infty)$ -- in particular $\gamma_j^{\tilde E}(\pm\infty)$ can be connected to $\tilde\gamma_1(\infty) = \gamma_{-1}^{\tilde E^1}(\infty)$. Now we show how $\tilde\gamma_1(-\infty)$ may be connected to $\gamma_j^{\tilde E}(\pm\infty)$. The sequence of $\tilde E$ may be a finite list of 1's, or has a first instance of taking some other symbol. If the sequence is all 1's then $\tilde\gamma_1(-\infty) = \gamma_{1}^{\tilde E}(-\infty)$ -- which can be connected to any of $\gamma_j^{\tilde E}(\pm\infty)$ as in case 1. If the sequence has $N$ ones and then some other symbol, $a$, then the $\tilde\gamma_1(-\infty) = \gamma_1^{\tilde E^N}(-\infty)$ and for a seam lifted to $\tilde E$, it either has some horocycles entirely contained in $\tilde E_{a}^N$ or equivalent to those of $\gamma_a^{\tilde E^N}$ -- which can all be made disjoint from horocycles of $\gamma_1^{\tilde E^N}(-\infty) = \tilde\gamma_1(-\infty)$.
\qed
\end{proof}

\section{Proof of theorem \ref{thm:syz}}
\label{sec:prf}

Now we consider the pair of pants, $\Sigma$, equipped with the non-positively curved reduced JM-metric. This metric is complete and (\cite{MPants} pg. 10) asymptotes to finite diameter cylinders around the collisions. In particular, $\Sigma$ satisfies the hypotheses of Lemma \ref{Lem:vis} by taking the 'seams' to be the collinear arcs. The proof of theorem \ref{thm:syz} consists of applying Lemma \ref{Lem:vis} to construct straight collision orbits realizing a given stutter free finite syzgy sequence, and then applying Toponogov's theorem\footnote{One form of this theorem states that a geodesic triangle in a non-positively curved manifold with interior angles $\alpha_i$ has $\alpha_1 + \alpha_2 + \alpha_3 \le\pi$ with equality only when the triangle bounds a region of zero curvature (see \cite{B} \S 1 B, in particular the consequence on pg. 8)} to show uniqueness. Finally one may obtain winding collision orbits by perturbing the straight collision orbits.

\begin{proof}[of theorem \ref{thm:syz}]
\smartqed
It is useful to first see how Lemma \ref{Lem:vis} is used to construct a straight collision orbit realizing the sequence $31$. We recall that -- due to the non-positive curvature -- two forwards or backwards asymptotic geodesics intersecting in a point are in fact the same geodesic.

Consider a fixed fundamental domain (centered in figure \ref{fig:31}). To obtain the first 3 in the sequence we can aim to cross the collinear arc 3 in this fundamental domain from 'top to bottom'. Then to obtain the following '1' in the sequence we want to pass next into the lower left fundamental domain. Now if there are to be no other syzygies in the sequence we must exit each of these fundamental domains down an appropriate leg: that is be backwards asymptotic to $x$ and forwards asymptotic to $y$ in the figure. By Lemma \ref{Lem:vis}, there exists a geodesic from $x$ to $y$. This geodesic cannot pass through the upper left or upper right regions without being trapped in them (since leaving these fundamental domains requires passing through a collinear arc asymptotic to $x$ -- forcing the geodesic to equal this collinear arc) nor can it pass through the lower right region since then it intersects the '2' collinear arc twice: which is not possible for two geodesics in a non-positively curved Hadamard manifold. Hence it passes from the centered fundamental domain to the lower left fundamental domain, and -- because it cannot cross any collinear arcs which it is asymptotic to without being equal to them -- realizes the syzygy sequence 31.

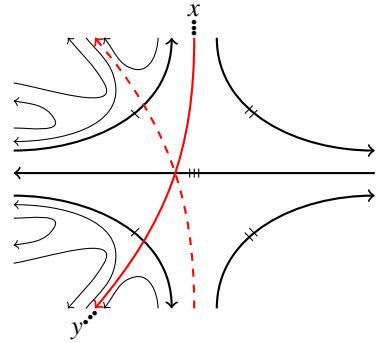
\begin{SCfigure}[2][!hbt]
    \begin{tikzpicture}[scale = .6]
    
    \node at (0,3.1) {\textbf{.}};
    \node at (0,3.22) {\textbf{.}};
    \node at (0,3.34) {\textbf{.}};
    \node at (0,3.6) {$x$};

  \draw[thick, color=black, <-] (-4,0) -- (4,0);
  
  
  \draw[color=black] (-.1,-.1) -- (-.1,.1);
  \draw[color=black] (0,-.1) -- (0,.1);
  \draw[color=black] (.1,-.1) -- (.1,.1);

  \draw [thick, color = black, ->] (-4,.5) to[out=0,in= -90] (-.5,3);
  \draw[color=black] (-1.4,1.4) -- (-1.22,1.22);

  \draw [thick, color = black, ->] (.5,3) to[out=-90,in= -180] (4,.5);
  \draw[color=black] (1.4,1.4) -- (1.22,1.22);
  \draw[color=black] (1.3,1.5) -- (1.12,1.32);

  \draw [thick, color = black, <-] (4,-.5) to[out=-180,in= 90] (.5,-3);
  \draw[color=black] (-1.4,-1.4) -- (-1.22,-1.22);

  \draw [thick, color = black, ->] (-4,-.5) to[out=0,in= 90] (-.5,-3);
  \draw[color=black] (1.4,-1.4) -- (1.22,-1.22);
  \draw[color=black] (1.3,-1.5) -- (1.12,-1.32);

  \draw [color = black, ->] (-4,-1) to[out=0,in= 60] (-3.1,-1.2) to[out=-120,in=10] (-4,-1.6);

  \draw [color = black, ->] (-4,-2) to[out=10,in= 60] (-2,-2) to[out=-120,in=50] (-2.8,-3);

  \draw [color = black, <-] (-2,-3) to[out=50,in= -170] (-1.3,-2.3) to[out=10,in=92] (-.8,-3);
  
  \draw [color = black, <-] (-4,-.7) to[out=0,in= 80] (-1.75,-2) to[out=-100,in=50] (-2.4,-3);
  
 \node at (-2.2,-3.1) {\textbf{.}};
\node at (-2.3,-3.2) {\textbf{.}};
\node at (-2.4,-3.3) {\textbf{.}};
\node at (-2.6,-3.5) {$y$};

  \draw [thick, color = red, ->] (0,3) to[out=-90,in=50] (-2.2,-3);

  \draw [color = black, ->] (-4,1) to[out=0,in= -60] (-3.1,1.2) to[out=120,in=-10] (-4,1.6);

  \draw [color = black, ->] (-4,2) to[out=-10,in= -60] (-2,2) to[out=120,in=-50] (-2.8,3);

  \draw [color = black, <-] (-2,3) to[out=-50,in= 170] (-1.3,2.3) to[out=-10,in=-92] (-.8,3);
  
  \draw [color = black, <-] (-4,.7) to[out=0,in= -80] (-1.75,2) to[out=100,in=-50] (-2.4,3);

  \draw [thick, dashed, color = red, ->] (0,-3) to[out=90,in=-50] (-2.2,3);

\end{tikzpicture}
\caption{Two straight collision geodesics (red) realizing the sequence $31$ (we use tick marks on the collinear arcs lifts in place of 1,2,3 to avoid cluttering the diagram). They are related by the symmetry of $\Sigma$ induced by a reflection in the plane containing the three bodies.}
\label{fig:31}
    \end{SCfigure}

One proceeds in the same way in general: associate to the finite syzygy sequence a corresponding finite sequence of fundamental domains in $H$ to pass through. In the first and last domains of this list, there will be one choice of end to shoot down, and then one invokes Lemma \ref{Lem:vis} to get a geodesic $\gamma$ connecting these two points of $H(\infty)$. Finally, using that forward asymptotic geodesics cannot intersect, nor can any two geodesics intersect more than once in $H$, we see that $\gamma$ indeed realizes the given syzygy sequence.

To see the orbit $\gamma$ is unique (up to the reflection symmetry), note that -- due to the finite diameter ends -- any other geodesic realizing the same syzygy sequence as $\gamma$ and passing through the same tiling sequence as $\gamma$ will be forward and backwards asymptotic to $\gamma$. It follows from Toponogov's theorem (\cite{B} pg. 8) that these two geodesics bound a flat strip, which contradicts that the JM-metric on $\Sigma$ is negative away from a discrete set.

Finally, we consider some winding collision orbits (see figure \ref{fig:winding}). Let $s_1...s_k$ be a finite stutter free syzygy sequence and $\gamma(t)$ a realizing geodesic. Varying the initial velocity $\dot\gamma(0)$ a sufficiently small amount from $\dot\gamma(0)$ one obtains -- since the ends asymptote to cylinders -- a geodesic $\hat\gamma$ which still begins and ends in the same collisions as $\gamma$ and -- by continuous dependence on initial conditions -- crosses $s_1...s_k$ before being sucked down the legs. However, as they are distinct and share a point, the lift of $\hat\gamma$ is not forwards or backwards asymptotic to the lift of $\gamma$, so as $\hat\gamma$ goes down the legs it will pick up the appropriate winding sequence.

\qed
\end{proof}

\begin{SCfigure}[2][!hbt]
    \begin{tikzpicture}[scale = .7]

  \draw[color=black] (-1,2) to [bend right=20] (3,.3);
  
  \node[below] at (.5,1.5) {3};

  \draw[color=black] (3,-.3) to [bend right=20] (-1,-2);
  
  \node at (.5,-1.1) {2};
  
  \draw[color=black] (-1.5,-1.9) to [bend right=20] (-1.5,1.9);
  
  \node at (-1.5,0) {1};
  
  \draw [thick, color = black] (-1.1,.1) to[out=70,in= -180] (0,.3) to[out= 0,in=-180] (3,.1);
  
  
  \draw [dashed, thick, color = black] (-1.1,.1) to[out=-110,in=-180] (0,-.3) to[out=0,in=-180] (3,-.1);

  \draw [thick, color = red] (-1.15,.22) to[out=70,in= 160] (0,.3) to[out= -20,in=-180] (2.5,-.34);
  
  \draw [dashed,thick,  color = red] (2.5, -.34) to[out=0,in= -100] (3,-.1);
  
  \draw [dashed, thick, color = red] (-1.15,.22) to[out=-110,in= 190] (0,-.1) to[out= 10,in=-180] (2.4,.355);
  
  \draw [thick, color = red] (2.4, .355) to[out=0,in= 170] (3,.2);

\end{tikzpicture}
\caption{Perturbing a straight collision orbit with syzygy sequence 1 (an isosceles solution) to get a winding orbit (red).}
\label{fig:winding}
    \end{SCfigure}
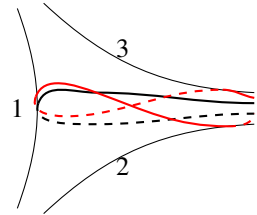

\begin{acknowledgement}
I thank Richard Montgomery for helpful comments as well as Josu\'e Mel\'endez for sharing some numerical experiments from \cite{Bill}, which inspired this work.
\end{acknowledgement}

\end{document}